\documentclass[12pt, reqno]{amsart}

\usepackage{amsmath, amssymb, latexsym} % Math packages
\usepackage{graphicx}                   % For including graphics
\usepackage{url}                        % Formatting URLs
\usepackage{xcolor}                     % Text color
\usepackage{comment}                    % For commenting out blocks of text
\usepackage{mathrsfs}                   % Script font
\usepackage[a4paper, margin=1in]{geometry} % Layout adjustments

\newcommand{\R}{{\mathbb R}}
\newcommand{\Q}{{\mathbb Q}}
\newcommand{\C}{{\mathbb C}}
\newcommand{\N}{{\mathbb N}}
\newcommand{\Z}{{\mathbb Z}}

\def\F {F^{(k)}}

\theoremstyle{plain}
\numberwithin{equation}{section}
\newtheorem{thm}{Theorem}[section]
\newtheorem{lemma}[thm]{Lemma}
\newtheorem{example}[thm]{Example}
\newtheorem{remark}[thm]{Remark}

\newtheorem{proposition}[thm]{Proposition}
\newtheorem{cor}[thm]{Corollary}

\begin{document}

\title[Fibonacci Numbers as Sums of Consecutive $k$-Generalized Fibonacci Numbers]{Fibonacci Numbers as Sums of Consecutive Terms in $k$-Generalized Fibonacci Sequence}

\author[R. Alvarenga]{Roberto Alvarenga}
\address{São Paulo State University (UNESP)\\
               São José do Rio Preto, Brazil.}
\email{roberto.alvarenga@unesp.br}

\author[A. P. Chaves]{Ana Paula Chaves}
\address{Federal University of Goáis (UFG)\\
               Goiânia, Brazil.}
\email{apchaves@ufg.br}

\author[M. E. Ramos]{Maria Eduarda Ramos}
\address{Federal University of Minas Gerais (UFMG)\\
               Belo Horizonte, Brazil.}
\email{madu-ramos@ufmg.br}

\author[M.  Silva]{Matheus  Silva}
\address{University of São Paulo (ICMC-USP), São Carlos, Brazil.}
\email{matheussilva1@usp.br}

\author[M. Sosa]{Marcos Sosa}
\address{Federal University of Latin-America Integration (UNILA)\\
                Foz do Iguaçu,  Brazil.}
\email{mes.garcete.2020@aluno.unila.edu.br}

\thanks{}

\begin{abstract} Let $(\F_n)_{n\geq -(k-2)}$ be the \(k\)-generalized Fibonacci sequence, defined as the linear recurrence sequence whose first \(k\) terms are \(0, 0,  \ldots, 0, 1\), and whose subsequent terms are determined by the sum of the preceding \(k\) terms. 
This article is devoted to investigating when the sum of consecutive numbers in the $k$-generalized Fibonacci sequence belongs to the Fibonacci sequence  \((F_n)_{n}\). Namely, given $d,k \in \N$, with $k \geq 3$,  our main theorem states that there are at most finitely many $n \in \N$ such that $F_n^{(k)} + \cdots + F_{n+d}^{(k)}$  
is a Fibonacci number. In particular, the intersection between the Fibonacci sequence and the $k$-generalized Fibonacci sequence is finite.
\end{abstract}

\maketitle

\section{Introduction}

Let \((F_n)_{n}\) be the well-known Fibonacci sequence defined by \(F_{n+2} = F_{n+1} + F_n\) for \(n \geq 0\), with initial conditions \(F_0 = 0\) and \(F_1 = 1\). Fibonacci numbers are famous for their fascinating properties and remarkable connections to various fields, such as nature, architecture, engineering, technology, computing, and more. For an overview of the Fibonacci sequence's history, properties, and intriguing applications and generalizations, see \cite{kalman-03, koshy-01, posamentier-07, vorobiev-02}. 

Among several of the mentioned generalizations, we consider the \emph{$k$-generalized Fibonacci sequence}, also known as \emph{$k$-step Fibonacci sequence} or \emph{$k$-bonacci sequence}, as introduced by Miles in \cite{miles-60}. For $k \in \N$ and $k \geq 2$, it is the sequence of natural numbers  $(\F_n)_{n\geq -(k-2)}$, given by 
\[ \F_{n+k} := \F_{n+k-1}+ \F_{n+k-2}+\cdots + \F_{n} \]
for $n \geq 0$, with  the initial conditions $\F_{-k+2} :=\F_{-k+1} := \cdots := \F_0=0$ and $\F_1 :=1$. We observe that for $k=2$, we retrieve the Fibonacci numbers, i.e.,\ $F_n^{(2)}=F_n$. For $k=3$, $F_n^{(3)}$ are known as the \emph{Tribonacci numbers} (see A000073 in the OEIS), and for $k=4$, $F_n^{(4)}$ are called the \emph{Tetranacci numbers} (see A000078 in OEIS). These $k$-generalized Fibonacci sequences have been widely investigated in the literature over the last three decades, particularly concerning Diophantine equations (see e.g., \cite{wolfram-98, chaves-marques-15, gomez-gomez-luca-24}). According to \cite{kessler-04}, these numbers appear in probability theory and certain sorting algorithms. In a recent paper, Alan and Altassan \cite{altassan-alan-23} introduced $k$-generalized \emph{tiny golden angles} by using the roots of the characteristic polynomial of the $k$-generalized Fibonacci sequence. The authors assert that these angles could improve the accuracy of MRI scans by optimizing the distribution of sampling points, leading to clearer images.

This work aims to investigate when the sum of consecutive 
$k$-generalized Fibonacci numbers equals a Fibonacci number. Namely, given $k,d \in \N$, with $k \geq 2$, we examine the existence of $m,n \in \N$ such that 
\begin{equation} \label{eq1}
     F_n^{(k)} + \cdots + F_{n+d}^{(k)} = F_{m}.
\end{equation}
The case $k=2$ is completely described in Proposition \ref{prop=casek=2}.  In \cite{marques-13}, the author address the case $d=0$, proving that
\[  (F_m)_m \cap (F_{n}^{(k)})_{n \geq -(k-2)} = 
\begin{cases}
    \{0,1,2,13\} & \text{ if } k=3, \\
     \{0,1,2,8\} & \text{ if } k>3. \\
\end{cases}
\]
For $d>0$, there are many examples where \eqref{eq1} holds. The first nontrivial case,  $k=3$ and $d=1$ yields the following instances 
\[ F_{-1}^{(3)} + F_{0}^{(3)} = F_0, \; 
F_{0}^{(3)} + F_{1}^{(3)} = F_1, \;
F_{0}^{(3)} + F_{1}^{(3)} = F_2, \;
F_{1}^{(3)} + F_{2}^{(3)} = F_3, \;
F_{2}^{(3)} + F_{3}^{(3)} = F_4.\]
In Example \ref{ex-final}, we verify that the possibilities for \eqref{eq1} in this case, namely for $k=3$ and $d=1$, are indeed finite. Moreover, we conjecture that the above are all the occurrences.  Others examples are: 
\[F_{1}^{(3)} + F_{2}^{(3)} + F_{3}^{(3)} + F_{4}^{(3)}  = F_6 \quad \text{ and } \quad    
F_{5}^{(5)} + F_{6}^{(5)} + F_{7}^{(5)} = F_{10}.\]
In the main theorem of this article, Theorem \ref{thm-main}, we show that for fixed $d, k\in \N$, with $k\geq 3$, equation \eqref{eq1} has at most finitely many solutions.

\section{Preliminaries}

In this section, we set up the notation and evoke basic facts that shall be applied henceforth. Moreover, we state and prove the lemmas needed to exhibit our main theorem. 

Recall the Binet formula for Fibonacci numbers:
\[ F_n = \frac{\alpha^n - \beta^n}{\sqrt{5}}, \quad \forall n \in \N, \ \mbox{ where } \ \alpha = \frac{1+ \sqrt{5}}{2} \ \mbox{ and } \ \beta = \frac{1- \sqrt{5}}{2}\]
Similarly, the $k$-generalized Fibonacci sequence also admits a Binet-type formula, which we introduce in the following. 
 
 Let 
\[  f_k(T) := T^k - T^{k-1} - \ldots - T - 1 \in \Z[T],\]
be the characteristic polynomial of $(F_{n}^{(k)})_n$.  It is well known that $f_k(T)$ is an irreducible polynomial over $\Q[T]$ with simple roots, cf.\ either \cite{miles-60} or \cite[Cor. 3.4 and 3.8]{wolfram-98}. Moreover, its roots $\alpha_1, \ldots, \alpha_k \in \C$, can be ordered such that
\[  3^{-k} < |\alpha_k| \leq \cdots  \leq |\alpha_2| < 1 < |\alpha_1|   \]
where actually $\alpha_1 \in \R$ and $ 2(1-2^{-k}) < \alpha_1 < 2$, cf.\  either \cite{miles-60} or \cite[Lemma 3.6]{wolfram-98}. The root $\alpha_1 \in \R$ is called the \emph{dominant root} of either $f_k(T)$ or $(F_{n}^{(k)})_n$. 
 
The mentioned Binet-type formula for the general term $F_{n}^{(k)}$ of the $k$-generalized Fibonacci sequence is given by the following proposition. 

\begin{proposition} \label{prop-binetfomulas} Let $k \in \N$,  $k\geq 2$, and $F_{n}^{(k)}$ be the $n$-th term in the $k$-generalized Fibonacci sequence. Then
    \[ F_{n}^{(k)} = \sum_{i=1}^{k} \frac{\alpha_i -1}{2+(k+1)(\alpha_i-2)} \alpha_{i}^{n-1}\]
where  $\alpha_1, \ldots, \alpha_k \in \C$ are the roots of the characteristic polynomial of  $(F_{n}^{(k)})_n$.   
\end{proposition}

\begin{proof}
   See \cite[Thm. 1]{dresden-du-14}.  
\end{proof}
    
\begin{remark} If $k=2$ in the above proposition, we obtain precisely the Binet formula for Fibonacci numbers. Moreover, there are others identities representing $F_{n}^{(k)}$ in terms of the roots of $f_k(T)$, see e.g., \ \cite[(2)'' in p. 749]{miles-60} and \cite{lee-01}.
\end{remark}

Next, we present the auxiliary results needed to prove our main theorem, listed in the order in which they appear in the text.

%%%%%%%%%%%%%%%%%%%%%%%%%%%%%%%%%%%%%%%%%%%%%%%%%%%%%%%%%%%%%%%
%%%%%%%%%%%%%%%%%%%%%%%%%%%%%%%%%%%%%%%%%%%%%%%%%%%%%%%%%%%%%%%

\begin{lemma} \label{lemma-1}
    Let $ k\in \N$, $k\geq 2$, and $\alpha_1 \in \R$ as  previously introduced. Then 
\[ N((k+1)\alpha_1-2k)= (-1)^{k} \Big( (2k)^{k-1}(k-1) - \sum_{j=0}^{k-2}(k+1)^{k-j}(2k)^j \Big),\]
where $N$ stands for the norm map of $\Q(\alpha_1)/\Q.$
\end{lemma}

\begin{proof} Let $A$ be the matrix, with respect to the basis $\{1, \alpha_1, \ldots \alpha_{1}^{k-1}\}$, of the $\Q$-linear transformation of $\Q(\alpha_1)$ given by the multiplication by $(k+1)\alpha_1-2k$. A straightforward calculation shows that

\[ A=\begin{pmatrix}
-2k  & 0 &  \cdots &  0 & k+1 \\
k+1 & -2k &  \cdots  &  0 & k+1 \\
0 & k+1 &  \cdots & 0  & k+1 \\
0 & 0 & \cdots & 0 &  k+1 \\
\vdots & \vdots & \vdots &  \vdots & \vdots \\
0 &  0 & \cdots & 0 &   k+1 \\
0 & 0 & \cdots & -2k & k+1 \\
0 & 0  &  \cdots & k+1 & -k+1 \\
\end{pmatrix}.\]
For $i=1, \ldots, k$, let  $A_i$ be the $i \times i$ matrix obtained from $A$ by removing the first $k-i$ rows and columns. In particular, $A_k =A$, $A_{2} = \begin{pmatrix}
    -2k & k+1 \\
    k+1 & -k+1
\end{pmatrix}$ and $A_1 = -k+1$. The determinant of $A_i$ can be calculated, for instance, through the Laplace expansion along its first row, which yields
\[\det(A_i)=  -2k \det(A_{i-1}) + (-1)^{i+1}(k+1)^{i}.\]
Let $\delta_i :=\det (A_i)$. Then
\[\delta_i =(-1)^{i+1}(k+1)^i - 2k \delta_{i-1}\]
where $\delta_0 := 1$ and $i=1, \ldots, k$.
We prove by induction that
\[ \delta_i= (-1)^{i} \Big( (2k)^{i-1}(k-1) - \sum_{j=0}^{i-2}(k+1)^{i-j}(2k)^j \Big) \]
for every $i=2, \ldots, k$. Indeed, it follows from the definition of $A_2$ that 
\[ \delta_2 = 2k(k-1) - (k+1)^2.\] 
Moreover, 
\begin{align*}
    \delta_{i+1} & = (-1)^{i}(k+1)^{i+1}- 2k \delta_i \\
    & = (-1)^{i}\Big((k+1)^{i+1} - (2k)^{i}(k-1) + \sum_{j=0}^{i-2}(k+1)^{i-j}(2k)^{j+1} \Big) \\
    & = (-1)^{i+1} \Big( (2k)^{i}(k-1) - \sum_{j=0}^{i-1}(k+1)^{i+1-j}(2k)^j \Big).
\end{align*}
Since $N((k+1)\alpha_1-2k)= \det(A)$, the lemma holds by replacing $i=k$ above. 
\end{proof}

%%%%%%%%%%%%%%%%%%%%%%%%%%%%%%%%%%%%%%%%%%%%%%%%%%%%%%%%%%%%%%%
%%%%%%%%%%%%%%%%%%%%%%%%%%%%%%%%%%%%%%%%%%%%%%%%%%%%%%%%%%%%%%%

Let $p \in \Z$ be a prime number. For $n \in \Q $, we use the standard notation $v_p(n)$ to represent its $p$-adic valuation. The following lemma will be helpful for proving the main theorem, as it will show that a certain linear form in logarithms is nonzero, allowing us to apply Baker's method.

\begin{lemma} \label{lemma-main} Let $d, k\in \N$, with $k\geq 3$, and $\alpha, \alpha_1 \in \R$ as  previously
introduced. Then 
\[ \frac{\alpha_1^{d+1} - 1 }{(k+1)\alpha_1-2k} \alpha_{1}^{n-1} 
- \frac{\alpha^m}{\sqrt{5}} \]
 is nonzero for all $m,n \in \N.$    
\end{lemma}

\begin{proof} Let $\Q(\alpha_1)$ be the number field obtained from adjoining $\alpha_1$ to $\Q$. From a previous discussion, it is known that $[\Q(\alpha_1):\Q]=k$. If
\[ \frac{\alpha_1^{d+1} - 1 }{(k+1)\alpha_1-2k} \alpha_{1}^{n-1} = \frac{\alpha^m}{\sqrt{5}}, \]
then $\alpha^m/\sqrt{5} \in \mathbb{Q}(\alpha_1)$. 
Since $[\Q(\alpha):\Q]=2$, we have $\Q(\alpha) = \Q(\alpha^m/\sqrt{5})$ and thus $\Q(\alpha) \subseteq \Q(\alpha_1).$  In particular, $k$ is even.

Let $N$ stands for the norm map of $\Q(\alpha_1)/\Q$. The previous identity yields 
\[ \frac{N(\alpha_1^{d+1} - 1 )}{N((k+1)\alpha_1-2k)} N(\alpha_{1}^{n-1}) = \frac{N(\alpha^m)}{N(\sqrt{5})}. \]
Let $N_{\Q(\alpha)/\Q}$ be the norm map of $\Q(\alpha)/\Q$. Since $N = N_{\Q(\alpha)/\Q}^k$, then 
\[N(\alpha) = (-1)^{k/2},\quad N(\sqrt{5}) = (-5)^{k/2} \quad  \text{and} \quad  N(\alpha_1)=(-1)^{k+1}.\] 
Replacing these values in the previous identity yields 
\[ N((k+1)\alpha_1-2k) = (-1)^{k/2 + (k+1)(n-1)} (-5)^{k/2}
N(\alpha_1^{d+1}-1). 
\]
Next, since $(k+1)\alpha_1-2k$ and $\alpha_1^{d+1}-1$ are algebraic integers, then both
$N((k+1)\alpha_1-2k)$ and $N(\alpha_1^{d+1}-1)$ are integers numbers. Therefore, $5^{k/2}$ must divide $N((k+1)\alpha_1-2k).$ Moreover, since $k$ is even and $k \geq 3$, we conclude that  $N((k+1)\alpha_1-2k)$ is divisible by $25$.

Lemma \ref{lemma-1} and the formula for the partial sum of geometric series yields 
\[ N((k+1)\alpha_1-2k)= (-1)^{k} \Big( (2k)^{k-1}(k-1) - (k+1)^2 \frac{(2k)^{k-1} - (k+1)^{k-1}}{k-1} \Big).\]
Let 
\[ \Delta(k) := \frac{N((k+1)\alpha_1-2k)}{(-1)^{k}} =(2k)^{k-1}(k-1) - (k+1)^2 \frac{(2k)^{k-1} - (k+1)^{k-1}}{k-1}.\]

We observe that if $k \not\equiv 1 \pmod{5}$, then $\Delta(k) \not\equiv 0 \pmod{25}$ if, and only if, $\Delta(k)(k-1) \not\equiv 0 \pmod{25}.$ Moreover, $\Delta(k)(k-1) \equiv \Delta(k+100)(k+99) \pmod{25}$. Indeed,
\begin{align*}
    \Delta(k+100)(k+99) & \equiv (2k+200)^{k+99}(k+99)^2 -(k+101)^2\big( (2k+200)^{k+99}-(k+101)^{k+99}\big)  \\[0.025cm] 
& \equiv  (2k)^{k-1}(2k)^{100}(k-1)^2  - (k+1)^2\big((2k)^{k+99}-(k+1)^{k+99})  \\[0.025cm]
& \equiv (2k)^{k-1}(k-1)^2 - (k+1)^2\big((2k)^{k-1}-(k+1)^{k-1} \big) \\[0.025cm]
& = \Delta(k)(k-1)
\end{align*}
where the equivalences are taken modulo $25$ and we have applied Euler's theorem to get the third one. It is straightforward to verify that  $\Delta(k) \not\equiv 0 \pmod{25}$ for any $2 \le k \le 100$ such that $k \not \equiv 1\pmod{5}$. Therefore, for $k \not \equiv 1\pmod{5}$, we have a contradiction with the fact that $N((k+1)\alpha_1-2k)$ is divisible by $25$.

Let $k \equiv 1 \pmod{5}$.  We assert that in this case,  
\[ 
v_5(k-1)=v_5(N((k+1)\alpha_1-2k)),\]
where $v_5$ is the $5$-adic valuation. Let $k = 5^a b +1$, where $a := v_5(k-1)$. Then
\begin{align*}
    \Delta(k) &  =   2^{k-1} \; (5^a b +1)^{k-1} \; (5^a b)
 - (5^a b+2)^2 \; \frac{(2)^{k-1} \; (5^a b+1)^{k-1}-(5^a b+2)^{k-1}}{5^a b}\\[0.055cm]  
& = 2^{k-1} \; (5^a b +1)^{k-1} \; (5^a b) - (5^a b+2)^2 \; \frac{\sum_{j=0}^{k-1}{k-1 \choose j}\; 2^{k-1-j} \; (5^a b)^j \;( 2^j - 1 ) }{5^a b} \\[0.055cm]
& = 2^{k-1} \; (5^a b +1)^{k-1} \; (5^a b) -  (5^a b+ 2)^2 \; \frac{\sum_{j=1}^{k-1}{k-1 \choose j}\; 2^{k-1-j} \; (5^a b)^j \;( 2^j - 1 ) }{5^a b}.
\end{align*}
Therefore, 
\begin{align*}
     \frac{\Delta(k)}{5^a} & = 2^{k-1} \; (5^a b +1)^{k-1} \; b - (5^a b+ 2)^2 \frac{\sum_{j=1}^{k-1}{k-1 \choose j}\; 2^{k-1-j} \; (5^a b)^j \;( 2^j - 1 ) }{5^{2a} b} \\
& \equiv  2^{k+1} b  \pmod{5} \\
& \not\equiv 0 \pmod{5}
\end{align*}
hence showing that $v_5(k-1)=v_5(N((k+1)\alpha_1-2k)).$

Since $5^{k/2}$ divides $\Delta(k)$, by previous discussion, 
\[ \frac{k}{2} \leq v_5(\Delta(k)) = v_5(k-1) \leq \log_5 (k-1)<\frac{k}{2},
\]
which also gives us a contradiction. 
\end{proof}

%%%%%%%%%%%%%%%%%%%%%%%%%%%%%%%%%%%%%%%%%%%%%%%%%%%%%%%%%%%%%
%%%%%%%%%%%%%%%%%%%%%%%%%%%%%%%%%%%%%%%%%%%%%%%%%%%%%%%%%%%%%

Next, we present a classical and fundamental tool for establishing the finiteness of solutions of \eqref{eq1}: a lower bound for linear forms in logarithms \emph{à la Baker}.

\begin{lemma} \label{lemma-matveev} Let $\gamma_1, \ldots, \gamma_n \in \overline{\Q} \setminus \{0,1\}$ and $b_1, \ldots, b_n \in \Q$. If 
$ \gamma_{1}^{b_1}  \cdots  \gamma_{n}^{b_n} - 1 \neq 0,$
then 
\[ (eB)^{-\lambda} <   \big| \gamma_{1}^{b_1}  \cdots  \gamma_{n}^{b_n} - 1 \big|, \]
where $B := \max\{|b_1|, \ldots, |b_n|\}$ and $\lambda \in \R$ is an effectively computable constant depending only on $\alpha_1, \ldots, \alpha_n \in \overline{\Q} \setminus \{0,1\}.$  
\end{lemma}

\begin{proof}
    See \cite{baker-75} and \cite{matveev-00} for an explicit description of $\lambda\in \R$.
    \end{proof}

%%%%%%%%%%%%%%%%%%%%%%%%%%%%%%%%%%%%%%%%%%%%%%%%%%%%%%%%%%%%%%
%%%%%%%%%%%%%%%%%%%%%%%%%%%%%%%%%%%%%%%%%%%%%%%%%%%%%%%%%%%%%%

\begin{lemma} \label{lemma6} Let $k \in \N_{>1}$ and $\alpha_1 \in \R$  be the dominant root of $f_k(T)$. Then, 
\begin{enumerate}
  \item  $ \alpha_{1}^{n-2} \leq F_{n}^{(k)} \leq \alpha_{1}^{n-1}$, and  \smallskip
  
    \item $F_{n}^{(k)} \leq 2^{n-1}$,
\end{enumerate}
 where $F_{n}^{(k)} $ is the $n$-th term of the $k$-generalized Fibonacci sequence. 
\end{lemma}

\begin{proof} The first item is \cite[Lemma 1]{bravo-luca-13}. The second one follows from the fact that \linebreak $\alpha_1 < 2$. 
    \end{proof}

%%%%%%%%%%%%%%%%%%%%%%%%%%%%%%%%%%%%%%%%%%%%%%%%%%%%%%%%%%%%%%

\section{The main theorem}

%%%%%%%%%%%%%%%%%%%%%%%%%%%%%%%%%%%%%%%%%%%%%%%%%%%%%%%%%%%%%%

In the following, we use the notation introduced in the previous section. We first consider the case $k=2$. 

\begin{proposition}\label{prop=casek=2} Let $d\in \N$ be fixed. Then 
\begin{equation} \label{eq-k=2}
    F_n + F_{n+1} + \cdots +F_{n+d} = F_m
\end{equation}
has solution if, and only if, either $d=0, 1$ or $d=2$ and $n=0$. 
\end{proposition}

\begin{proof} 
As noticed before, if $k=2$, then $(F_{n}^{(2)})_n$ is simply $(F_n)_n$, the Fibonacci sequence. If $d=0,1$, the definition of $(F_n)_n$ and its recurrence, ensure that the identity \eqref{eq-k=2} is satisfied for all $n \in \N$, with $m=n,n+2$, respectively. Moreover, if $d=2$, then \eqref{eq-k=2} is satisfied for $n=0$ and $m=3$.  

On the other hand, given $d,n \in\N$, suppose 
\[ F_n + \cdots + F_{n+d} = F_m,\] 
for some $m\in \N$. We observe that $F_m = F_{n+d+2} - F_{n+1}$. Then, $F_m < F_{n+d+2}$, for all $d,n \in \N$. Moreover, if $d=1$, we are done, thus we might suppose $d>1$ in the following. 

Let $n=0$. Then $F_m = F_{d+2} -1$. If $F_m < F_{d+1}$, then $F_d \leq 0$, which implies $d=0$. If $F_m \geq F_{d+1}$, since $F_m < F_{d+2}$, then $F_m = F_{d+1}$. Hence, $F_d =1$, which implies $d=2$ and $m=3$. 

Let $n \geq 1$. Then $n+d > n+1 > 2$. Thus $F_{n+d} > F_{n+1}$ and
\[
  F_m :=  F_{n+d+2} - F_{n+1} > F_{n+d+2} - F_{n+d} = F_{n+d+1}.
   \]
Thereby, 
\[ F_{n+d+1} < F_m < F_{n+d+2},\] 
i.e., \  $F_m$ lies strictly between two consecutive Fibonacci numbers. Therefore, $F_m$ can not be a Fibonacci number for $d,n > 1$. 
    \end{proof}

The central result of this work is stated below. 

\begin{thm} \label{thm-main}   
Let $d, k\in \N$ be fixed, with $k\geq 3$. Then there are (at most) finitely many $m,n \in \Z$ such that 
\begin{equation}\tag{1.1} \label{eq-main}
    F_n^{(k)}+F_{n+1}^{(k)}+ \cdots +F_{n+d}^{(k)} = F_m
\end{equation} 
\end{thm}

%%%%%%%%%%%%%%%%%%%%%%%%%%%%%%%%%%%%%%%%%%%%%%%%%%%%%%%%%%%%%%

\begin{proof}
Suppose that there are infinitely many integers $m,n \in \Z$ satisfying identity \eqref{eq-main}. Hence, there are sequences of positive integers $(x_{\ell})_{\ell}$ and $(y_{\ell})_{\ell}$ such that 
\[ F_{x_{\ell}}^{(k)}+F_{x_{\ell}+1}^{(k)}+ \cdots +F_{x_{\ell}+d}^{(k)} = F_{y_{\ell}} \]
for all $\ell \in \N.$ 
Applying Proposition \ref{prop-binetfomulas} in the previous identity yields
\[ \sum_{j=0}^{d} \sum_{i=1}^{k} \frac{\alpha_i -1}{2+(k+1)(\alpha_i-2)} \alpha_{i}^{x_{\ell}+j-1} = \frac{\alpha^{y_{\ell}} - \beta^{y_\ell}}{\sqrt{5}}.\]
Hence, 
\[  c_1\sum_{j=0}^{d} \alpha_{1}^{x_{\ell}+j-1} - \frac{\alpha^{y_{\ell}}}{\sqrt{5}} = 
- \sum_{j=0}^{d} \sum_{i=2}^{k} \frac{\alpha_i -1}{2+(k+1)(\alpha_i-2)} \alpha_{i}^{x_{\ell}+j-1} - \frac{\beta^{y_\ell}}{\sqrt{5}}, \]
where $c_1 := \frac{\alpha_1 -1}{2 + (k+1)(\alpha_1 -2)}.$ Thus, 
\begin{align*}
    \Big|c_1\sum_{j=0}^{d} \alpha_{1}^{x_{\ell}+j-1} - \frac{\alpha^{y_{\ell}}}{\sqrt{5}}\Big| & \leq 
    \sum_{j=0}^{d} \sum_{i=2}^{k} \Big|\frac{\alpha_i -1}{2+(k+1)(\alpha_i-2)} \alpha_{i}^{x_{\ell}+j-1} \Big| + \frac{\big|\beta\big|^{y_\ell}}{\sqrt{5}}\\
    & \leq c_2 \sum_{j=0}^{d} \big| \alpha_{2} \big|^{x_{\ell}+j-1} + \frac{\big|\beta\big|^{y_\ell}}{\sqrt{5}},
\end{align*}
where $c_2 := \sum_{i=2}^{k} \Big|\frac{\alpha_i -1}{2+(k+1)(\alpha_i-2)}\Big|$, and the last inequality is obtained from the fact that 
$|\alpha_2| \geq |\alpha_i|$, for $i=3, \ldots, k.$ 

Next, we observe that $\sum_{j=0}^{d} \alpha_{1}^{x_{\ell}+j-1}$ is the sum of the first $d+1$ terms in the geometric sequence given by the initial value $\alpha_1^{x_{\ell}-1}$ and common ratio $\alpha_1$. Then, 
\[\Big|c_1\sum_{j=0}^{d} \alpha_{1}^{x_{\ell}+j-1} - \frac{\alpha^{y_{\ell}}}{\sqrt{5}}\Big|  =
\Big| \frac{\alpha_{1}^{d+1}-1}{2+(k+1)(\alpha_1-2)} \alpha_{1}^{x_{\ell}-1} - \frac{\alpha^{y_{\ell}}}{\sqrt{5}}\Big|.
\]
Hence, 
\[ \Big|  c_3 \alpha_{1}^{x_{\ell}-1} - \frac{\alpha^{y_{\ell}}}{\sqrt{5}}\Big| \leq 
c_2 \sum_{j=0}^{d} \big| \alpha_{2} \big|^{x_{\ell}+j-1} + \frac{\big|\beta\big|^{y_\ell}}{\sqrt{5}},\]
where $c_3 := \frac{\alpha_{1}^{d+1}-1}{2+(k+1)(\alpha_1-2)}.$
Multiplying both sides of the above inequality by $\sqrt{5}/\alpha^{y_{\ell}}$ yields
\[ \Big|  c_3  \; \alpha_{1}^{x_{\ell}-1}  \alpha^{-y_{\ell}} \sqrt{5} - 1 \Big| \leq  
c_2  \; \alpha^{- y_{\ell}} \sqrt{5} \sum_{j=0}^{d} \big| \alpha_{2} \big|^{x_{\ell}+j-1} + \big|\beta\big|^{2 y_\ell}.\]
From Lemma \ref{lemma-main}, 
\begin{equation} \label{eq-forexample}
     c_3   \alpha_{1}^{x_{\ell}-1}  \alpha^{-y_{\ell}} \sqrt{5} - 1 \neq 0.
\end{equation}
Thus, Lemma \ref{lemma-matveev} yields the existence of $\lambda \in \R$ such that  
\[ (B e)^{-\lambda} \leq  c_2  \; \alpha^{- y_{\ell}} \sqrt{5} \sum_{j=0}^{d} \big| \alpha_{2} \big|^{x_{\ell}+j-1} + \big|\beta\big|^{2 y_\ell},\] 
 where $B := \max\{x_{\ell}-1, y_{\ell}, 1\}$. 
Since 
\[ F_{y_{\ell}} = \sum_{j=0}^{d} F_{x_{\ell}+j}^{(k)} \geq F_{x_{\ell}}^{(k)} \geq F_{x_{\ell}},\] 
then $y_{\ell} \geq x_{\ell}$ and hence $B = y_{\ell}.$ Follows from Lemma \ref{lemma6} that
\[  \alpha^{y_{\ell}-2} \leq F_{y_{\ell}} = \sum_{j=0}^{d} F_{x_{\ell}+j}^{(k)} \leq 
\sum_{j=0}^{d} 2^{x_{\ell}+j-1}  
< 2^{x_{\ell}+d} < \alpha^{2 x_{\ell}+2 d}.
\]
Thus, $y_{\ell} \leq 2 x_{\ell}+2 d + 2.$ Hence, 
\[((2 x_{\ell}+2 d + 2)e)^{-\lambda} \leq 
(y_{\ell}e)^{-\lambda} \leq 
c_2  \; \alpha^{- y_{\ell}} \sqrt{5} \sum_{j=0}^{d} \big| \alpha_{2} \big|^{x_{\ell}+j-1} + \big|\beta\big|^{2 y_\ell}.
\]
Multiplying the previous inequality by $(2 x_{\ell}+2 d + 2)^{\lambda}$ yields 
\[e^{-\lambda} \leq
(2 x_{\ell}+2 d + 2)^{\lambda} \left( c_2  \; \alpha^{- y_{\ell}} \sqrt{5} \sum_{j=0}^{d} \big| \alpha_{2} \big|^{x_{\ell}+j-1} + \big|\beta\big|^{2 y_\ell} \right).
\]
Since $|\alpha_2|< 1$ and $|\beta|< 1$, letting $n$ goes to infinity implies $e^{-\lambda} \leq 0$, which is a contradiction. This completes the proof.  \end{proof}

%%%%%%%%%%%%%%%%%%%%%%%%%%%%%%%%%%%%%%%%%%%%%%%%%%%%%%%%%%%%%
%%%%%%%%%%%%%%%%%%%%%%%%%%%%%%%%%%%%%%%%%%%%%%%%%%%%%%%%%%%%%

\begin{cor} Let $k \in \N$, $k \geq 3$. Then the intersection between the Fibonacci sequence and the $k$-generalized Fibonacci sequence is finite.   
\end{cor}

\begin{proof}
    Let $d=0$ in the previous theorem. We note that this corollary was previously established in \cite{marques-13}. 
\end{proof}

We end this article with a practical example for Theorem \ref{thm-main}. 

\begin{example} \label{ex-final} Consider the following Diophantine equation in $m,n \in \Z$, with $n \geq -1$ and $m \geq 0$: 
\begin{equation}\label{eq-ex}
    F_{n}^{(3)} + F_{n+1}^{(3)} = F_m.
\end{equation}
We claim that there are finitely many $(n,m)$ being a solution for \eqref{eq-ex}. Moreover, we conjecture that $(n,m)$ is a solution for \eqref{eq-ex} if, and only if, 
\[ (n,m) \in \big\{ (-1,0), (0,1), (0,2), (1,3), (2,4) \big\}.\]

In order to prove our claim, we follow the proof of Theorem \ref{thm-main} to obtain an upper bound for $n$ and $m$. 

In what follows, we use the notation of Lemma \ref{lemma-matveev}. Our first task is to explicitly 
 exhibit $\lambda \in \R$. In order to obtain such explicit description, we follow \cite{matveev-00}. Let $K$ be a degree $\ell$ number field containing  
$\gamma_1, \ldots, \gamma_k \in \R \cap \overline{\Q} \setminus \{0,1\}$. Then 
\[ \lambda := C_{k,\ell} \prod_{i=1}^k A_i,\]
where
\[ C_{k,\ell} := 1.4\cdot 30^{k+3} \cdot k^{4.5} \cdot \ell^2 (1+ \log(\ell))\]
and $A_1, \ldots, A_k \in \R$ are defined as follows. For $\gamma \in \overline{\Q}$ of degree $d$, let
\[ h(\gamma) :=  \frac{1}{d} \left( \log(a_0) + \sum_{i=1}^{d} \log(\max\{\gamma^{(i)},1\}) \right)\]
be the logarithmic height of $\gamma$, where $ p(T) := a_0 \prod_{i=1}^{d} (T - \gamma^{(i)}) \in \Z[T]$
is the minimal primitive polynomial of $\gamma$ having positive leading coefficient. Thus 
\[ A_i \geq \max\{\ell h(\gamma_i), |\log(\gamma_i)|, 0.16\}, \quad i=1,\ldots, k.\]
In our case, follows from \eqref{eq-forexample} in the proof of Theorem \ref{thm-main} that 
\[ c_3\; \alpha_{1}^{n-1} \alpha^{-m} \sqrt{5} -1 \neq 0\]
where $\alpha_1 \in \R$ is the dominant root of $(F_{n}^{(3)})_{n}$ and $c_3 := (\alpha_{1}^2-1)/(4 \alpha_1 -2).$ Thus, we might assume $K = \Q(\alpha_1, \sqrt{5})$ and, thereby, $\ell = 6$. 
Since $k=4$, $C_{4,6} = 1.57 \cdot 10^{15}.$ Moreover, as in the proof of Theorem \ref{thm-main}, $B = m$.  

Let us denote $\gamma_1  = c_3, \gamma_2 = \alpha_1, \gamma_3 = \alpha$ and $\gamma_4 = \sqrt{5}.$ We are left to calculate $A_1, A_2, A_3, A_4$.
The minimal primitive polynomial of $c_3$  with positive leading coefficient is given by
\[ 26 T^3 - 20 T^2 + 6T -1 = 26 (T-c_3) (T- z_1) (T- z_2) \in \Z[T],\]
with $|z_1| = |z_2| = 0.29$.  Thus $h(c_3) = 1.07$ and we might assume $A_1 = 6.42$. Since $\alpha_1$ is the dominant root of $(F_{n}^{(3)})_n$, thus $h(\alpha_1) = 0.2$ and, thereby, we might assume $A_2 = 1.22$. Analogously, $\alpha$ is the dominant root of $(F_m)_m$, thus $h(\alpha)=0.24$ and we might assume $A_3 = 1.44$. Finally, $h(\sqrt{5}) = 0.8$ and we might assume $A_4 = 4.83$.  Therefore, $\lambda = 85.53 \cdot 10^{15}.$

As in the end of the proof of Theorem \ref{thm-main},
\[ e^{-\lambda} \leq (2n+4)^{\lambda} \left( c_2 \alpha^{-m} \sqrt{5} (|\alpha_2|^{n-1} + |\alpha_2|^n) + 0.62^{2m} \right),\]
where $\alpha_2, \alpha_3 \in \C$ are the (conjugate) non-dominant roots of $(F_{n}^{(3)})_n$ and $c_2 = \sum_{i=2}^{3} \left| \frac{\alpha_i - 1}{4\alpha_i -2} \right|$. 
Since $m \geq n$ and $|\alpha_2| = |\alpha_3| < 0.74$, follows from the last inequality that 
\[ 1 < ((2n+4) e)^{\lambda} \left( 1.57 \alpha^{-n} (0.74^{n-1} + 0.74^n) + 0.62^{2n} \right).\]
We apply the logarithm and simplify the above inequality, as follows:
\begin{align*}
    0  & < \lambda (1 + \ln(2n+4)) + \ln\left( 1.57 \;\alpha^{-n} ( 0.74^{n-1} + 0.74^n) + 0.62^{2n} \right) \\
    & < \lambda (1 + \ln(2n+4)) + \ln\left( 3.14 \;\alpha^{-n}  0.74^{n-1} + 0.62^{2n} \right)\\
    & = \lambda (1.7 + \ln(n+2)) + \ln\left( 4.24 (0.46)^n + 0.38^{n} \right) \\
    & < 1.46 \cdot 10^{17} + 85.53 \cdot 10^{15} \ln(n+2) - n (0.78). 
\end{align*}
To summarize, if $n \in \N$ is such that \eqref{eq-ex} holds, then $n$ must verify the last inequality. 
We thus apply a Python script to determine that, for $n=9223372036854775808$, the abovementioned inequality does not hold. The script increases $n$ exponentially to optimize runtime until such a value is found. This approach speeds up the computation but does not guarantee that the identified $n$ is the smallest one for which the inequality fails. 

As in the proof of Theorem \ref{thm-main},  $m \leq 2 n +4.$ Therefore, the above discussion provides an upper bound for $m$ and $n$ as solutions of  \eqref{eq-ex}, yielding in finitely many solutions. 
\end{example}

%%%%%%%%%%%%%%%%%%%%%%%%%%%%%%%%%%%%%%%%%%%%%%%%%%%%%%%%%%%%%
%%%%%%%%%%%%%%%%%%%%%%%%%%%%%%%%%%%%%%%%%%%%%%%%%%%%%%%%%%%%%

\section*{Acknowledgement}
This is the first article resulting from the research project: \textit{Fibonacci Journey}. The project had its beginning through the research program \textit{Jornadas de Pesquisa em Matemática do ICMC 2023}, held at ICMC-USP. The authors are very grateful for the hospitality and support of ICMC-USP, as well as to the organizers of this remarkable project. This work was financed, in part, by the São Paulo Research Foundation (FAPESP), Brasil. Process Numbers 2013/07375-0 and 2022/09476-7.

%\bibliographystyle{alpha}
%\bibliography{fibonacci}  

\begin{thebibliography}{LLKS01}
	
	\bibitem[AA23]{altassan-alan-23}
	A.~Altassan and M.~Alan.
	\newblock Almost repdigit $k$-{Fibonacci} numbers with an application of
	k-generalized fibonacci sequences.
	\newblock {\em Mathematics}, 11(2), 2023.
	
	\bibitem[Bak75]{baker-75}
	A.~Baker.
	\newblock A sharpening of the bounds for linear forms in logarithms. {III}.
	\newblock {\em Acta Arith.}, 27:247--252, 1975.
	
	\bibitem[BL13]{bravo-luca-13}
	J.~J. Bravo and F.~Luca.
	\newblock On a conjecture about repdigits in {{\(k\)}}-generalized {Fibonacci}
	sequences.
	\newblock {\em Publ. Math. Debr.}, 82(3-4):623--639, 2013.
	
	\bibitem[CM15]{chaves-marques-15}
	A.~P. Chaves and D.~Marques.
	\newblock A {Diophantine} equation related to the sum of powers of two
	consecutive generalized {Fibonacci} numbers.
	\newblock {\em J. Number Theory}, 156:1--14, 2015.
	
	\bibitem[DD14]{dresden-du-14}
	G.~P.~B. Dresden and Z.~Du.
	\newblock A simplified {Binet} formula for {{\(k\)}}-generalized {Fibonacci}
	numbers.
	\newblock {\em J. Integer Seq.}, 17(4):article 14.4.7, 9, 2014.
	
	\bibitem[GGL24]{gomez-gomez-luca-24}
	C.~A. G{\'o}mez, J.~C. G{\'o}mez, and F.~Luca.
	\newblock A {Diophantine} equation with powers of three consecutive
	{{\(k\)}}-{Fibonacci} numbers.
	\newblock {\em Result. Math.}, 79(4):21, 2024.
	\newblock Id/No 136.
	
	\bibitem[KM03]{kalman-03}
	D.~Kalman and R.~Mena.
	\newblock The {Fibonacci} numbers -- exposed.
	\newblock {\em Math. Mag.}, 76(3):167--181, 2003.
	
	\bibitem[Kos01]{koshy-01}
	T.~Koshy.
	\newblock {\em Fibonacci and {Lucas} numbers with applications. {Volume} {I}}.
	\newblock Pure Appl. Math., Wiley-Intersci. Ser. Texts Monogr. Tracts. New
	York, NY: Wiley, 2001.
	
	\bibitem[KS04]{kessler-04}
	D.~Kessler and J.~Schiff.
	\newblock A combinatoric proof and generalization of {Ferguson}'s formula for
	{{\(k\)}}-generalized {Fibonacci} numbers.
	\newblock {\em Fibonacci Q.}, 42(3):266--273, 2004.
	
	\bibitem[LLKS01]{lee-01}
	G.-Y. Lee, S.-G. Lee, J.-S. Kim, and H.-K. Shin.
	\newblock The {Binet} formula and representations of {{\(k\)}}-generalized
	{Fibonacci} numbers.
	\newblock {\em Fibonacci Q.}, 39(2):158--164, 2001.
	
	\bibitem[Mar13]{marques-13}
	D.~Marques.
	\newblock The proof of a conjecture concerning the intersection of
	{{\(k\)}}-generalized {Fibonacci} sequences.
	\newblock {\em Bull. Braz. Math. Soc. (N.S.)}, 44(3):455--468, 2013.
	
	\bibitem[Mat00]{matveev-00}
	E.~M. Matveev.
	\newblock An explicit lower bound for a homogeneous rational linear form in
	logarithms of algebraic numbers. {II}.
	\newblock {\em Izv. Math.}, 64(6):1217--1269, 2000.
	
	\bibitem[MJ60]{miles-60}
	E.~P. Miles~Jr.
	\newblock Generalized {Fibonacci} numbers and associated matrices.
	\newblock {\em Am. Math. Mon.}, 67:745--752, 1960.
	
	\bibitem[PL07]{posamentier-07}
	A.~S. Posamentier and I.~Lehmann.
	\newblock {\em The (fabulous) {Fibonacci} numbers. {With} an afterword by
		{Herbert} {A}. {Hauptman}}.
	\newblock Amherst, NY: Prometheus Books, 2007.
	
	\bibitem[Vor02]{vorobiev-02}
	N.~N. Vorobiev.
	\newblock {\em Fibonacci numbers. {Transl}. from the {Russian} by {Mircea}
		{Martin}}.
	\newblock Basel: Birkh{\"a}user, 2002.
	
	\bibitem[Wol98]{wolfram-98}
	D.~A. Wolfram.
	\newblock Solving generalized {Fibonacci} recurrences.
	\newblock {\em Fibonacci Q.}, 36(2):129--145, 1998.
	
\end{thebibliography}

\end{document}